\documentclass{amsart}
\usepackage{graphicx} 

\usepackage{amsmath}
\usepackage{amssymb}
\usepackage{amsbsy}
\usepackage{amssymb}

\usepackage[hidelinks]{hyperref}
\usepackage{cleveref} 

\usepackage{mathdots} 
\usepackage{enumerate}

\usepackage{mathrsfs} 

\usepackage{txfonts}  
\usepackage{upgreek}
\usepackage{dutchcal}
\usepackage{frcursive}
\usepackage{calligra}
\usepackage[T1]{fontenc}

\usepackage{verbatim}

\usepackage{stmaryrd} 

\usepackage{fancyhdr}  

\usepackage{fullpage}

\usepackage{comment}

\usepackage{algorithm}

\usepackage{algpseudocode}

\usepackage{tikz}

\usetikzlibrary{arrows,positioning}

\usepackage{makecell}

\usetikzlibrary{datavisualization.formats.functions}

\usepackage{pgfplots}

\usepackage{graphicx,amssymb}

\usepackage{ifthen}

\usepackage{epsfig}

\newtheorem{theo}{Theorem}

\newtheorem{coro}{Corollary}

\newtheorem{clai}{Claim}

\renewcommand{\le}{\leqslant}
\renewcommand{\leq}{\leqslant}

\renewcommand{\ge}{\geqslant}
\renewcommand{\geq}{\geqslant}

\newcommand{\vep}{\varepsilon}

\newcommand{\yn}{\ensuremath{\{y_n\}_{n=1}^\infty}}

\newcommand{\zn}{\ensuremath{\{z_n\}_{n=1}^\infty}}

\newcommand{\norm}[1]{\| #1\|}

\newcommand{\cof}{\textrm{cof}}

\usepackage{amsmath}
\usepackage{amssymb}
\usepackage{amsbsy}
\usepackage{thmtools}
\usepackage{amsthm}
\usepackage[hidelinks]{hyperref}
\usepackage{cleveref} 

\usepackage{mathdots} 
\usepackage{enumerate}

\title[AMUC without AUC renorming]{The Kadets--Werner modification of Bourgain--Rosenthal space is asymptotically midpoint uniformly convex}
\author{Florent Baudier}
\address{Department of Mathematics, Texas A\&M University, College Station, TX 77843-3368, U.S.A.}
\email{florent@tamu.edu}
\urladdr{https://people.tamu.edu/~florent/}

\subjclass[2020]{46B03,46B10,46B20}

\keywords{asymptotic midpoint uniform convexity, asymptotic uniform convexity, Daugavet Property, asymptotic geometry of Banach spaces}

\date{\today}

\begin{document}
\begin{abstract}
Let $X_{\mathsf{KW}}$ be the Kadets--Werner modification of a closed subspace of $L_1$ constructed by Bourgain and Rosenthal in 1980. In this short note, it is shown that $X_{\mathsf{KW}}$ is asymptotically midpoint uniformly convex. Since $X_{\mathsf{KW}}$ has the Daugavet Property, it fails the Point of Continuity Property and thus does not admit an equivalent norm that is asymptotically uniformly convex. Therefore, the previously known properties of $X_{\mathsf{KW}}$ and the new geometric observation answer, in the negative, the question of Dilworth, Kutzarova, Randrianarivony, Revalski and Zhivkov whether asymptotic midpoint uniform convexity and asymptotic uniform convexity are isomorphically equivalent and a question of Perreau whether asymptotic midpoint uniform convexity implies the Point of Continuity Property.
\end{abstract}
\maketitle

\textbf{AI Disclosure:} The proof of the new observation was generated by OpenAI’s GPT-6 Astra during a chat initiated by the author in an attempt to understand what spaces admit an equivalent norm that is asymptotically sprawling uniformly convex. This notion was recently introduced in \cite{BLP} and shown to be equivalent to asymptotic midpoint uniform convexity. The author has checked the mathematical correctness of the proof and edited it for clarity, presentation, and context. The author acknowledges OpenAI for providing platform access via the ChatGPT for Academic Researchers program.

\section{Introduction}

Given a Banach space $(X,\norm{\cdot}_X)$, its modulus of \emph{asymptotic uniform convexity} is defined, for $t\in(0,\infty)$, by
\begin{equation*}
 \overline\delta_X(t) : =\inf_{x\in S_X}\ \sup_{Y\in \cof(X)}\ \inf_{y\in S_Y}
       \bigl(\norm{x+ty}_X-1\bigr),
\end{equation*}
where $\cof(X)$ denotes the set of closed finite-codimensional subspaces of $X$.
The space $X$, or more precisely its norm, is called \emph{asymptotically uniformly convex} if $\overline\delta_X(t)>0$ for every $t>0$. The history and importance of this fundamental geometric notion, dating back to Milman \cite{Milman71} and popularized in \cite{JLPS02}, are discussed in \cite{BL}. More recently, a weakening of this notion was introduced by Dilworth, Kutzarova, Randrianarivony, Revalski and Zhivkov \cite{Detal}. A Banach space $(X,\norm{\cdot}_X)$ is said to be \emph{asymptotically midpoint uniformly convex} if $\tilde\delta_X(t)>0$ for every $t>0$, where this time 
\begin{equation*}
 \tilde\delta_X(t) :=\inf_{x\in S_X}\ \sup_{Y\in \cof(X)}\ \inf_{y\in S_Y}
       \bigl(\max\{\norm{x+ty}_X,\norm{x-ty}_X\}-1\bigr).
\end{equation*}

It is elementary to check that $\tilde\delta_X\ge\overline\delta_X$, and hence every asymptotically uniformly convex norm is asymptotically midpoint uniformly convex. A renorming of $\ell_2$ was constructed in \cite{Detal} to distinguish the two notions isometrically. It was also shown there that every asymptotically midpoint uniformly convex Banach space with an unconditional basis admits an equivalent norm that is asymptotically uniformly convex, and naturally the question whether this remains true in general was asked by Dilworth, Kutzarova, Randrianarivony, Revalski and Zhivkov in \cite{Detal} and is listed as Problem 14 in \cite{BL}. In \cite{Perreau}, Perreau studied the original as well as the weak$^*$-versions (in dual spaces) of the notions above. He showed that the weak*-versions of AMUC and AUC are equivalent up to renorming in the class of duals of separable spaces with a weak*-unconditional asymptotic structure, but not in general, as witnessed by the dual of the James tree space.

Perreau specifically asked whether asymptotic midpoint uniform convexity implies the Point of Continuity Property. Recall that a Banach space has the Point of Continuity Property if for every nonempty bounded closed subset $C$, the identity map from $C$ equipped with the relative weak topology to $C$ but this time equipped with the relative norm topology, has a point of continuity. Note that the Point of Continuity Property is preserved under equivalent renorming. For our purposes, this property is best seen from the perspective of fragmentability indices. Indeed, $X$ has the Point of Continuity Property if and only if for every nonempty bounded closed subset $C$ of $X$, $C$ is weakly fragmentable, meaning that for all $\vep>0$ there is a nonempty relatively weakly open subset of $C$ of diameter at most $\vep$. 

A property going somewhat in the opposite direction from the Point of Continuity Property is the Daugavet Property. Originally defined as an additive property of operator norms, namely $\norm{\mathrm{Id}+T}=\norm{\mathrm{Id}} +\norm{T}$ must hold for every rank one operator, the Daugavet Property has been characterized in terms of a property of slices in \cite{KSSW} and of weakly open sets by Shvydkoy \cite{Shvydkoy}. In particular, it follows from \cite{Shvydkoy} that in a Banach space with the Daugavet property, every nonempty relatively weakly open subset of its unit ball has diameter $2$ (see for instance \cite[Theorem 2.1]{ALN}). Therefore, it immediately follows that a Banach space with the Daugavet Property must fail the Point of Continuity Property. It is worth pointing out that a space with the Daugavet Property does not have an unconditional basis \cite{Kadets}; in fact, it does not even embed into a space with one \cite{KSSW}.

It is easy to see that asymptotic uniform convexity implies the Point of Continuity Property once we have the reformulation of asymptotic uniform convexity in terms of a ``small weakly open set'' property. A Banach space is asymptotically uniformly convex if and only if for every $\vep>0$ there is $\delta>0$ such that for every point $x\in B_X$ with $\norm{x}\geq 1-\delta$, there is a relatively weakly open subset of $B_X$ containing $x$ with diameter at most $\vep$.

To summarize the above discussion, a Banach space with the Daugavet Property automatically fails the Point of Continuity Property and thus does not admit an equivalent asymptotically uniformly convex norm. Therefore, a candidate for an example of a Banach space that is asymptotically midpoint uniformly convex but lacks any asymptotically uniformly convex renorming is a Banach space with the Daugavet Property. Classic examples of Banach spaces with the Daugavet Property are $C([0,1])$ or $L_1([0,1])$, but these two examples do not have asymptotically midpoint uniformly convex renormings (see \cite{Baudieretal} for instance). Another example of a space with the Daugavet Property is $X_{\mathsf{KW}}$, the Kadets--Werner modification \cite{KW} of the Bourgain--Rosenthal closed subspace of $L_1$ \cite{BR}. The Bourgain--Rosenthal construction takes the closure of the union of some well-chosen increasing sequence of finite-dimensional subspaces of $L_1$, to provide an example of a (separable) Banach space failing the Infinite Tree Property as well as the Radon--Nikod\'ym Property. The modification of Kadets and Werner provides us with a (separable) Banach space with the Schur Property and the Daugavet Property (and still failing the Infinite Tree Property and the Radon--Nikod\'ym Property). 

In the next section, it is shown that $X_{\mathsf{KW}}$ is asymptotically midpoint uniformly convex, thereby providing negative answers to the question of Dilworth, Kutzarova, Randrianarivony, Revalski and Zhivkov and to the question of Perreau.

\section{Asymptotic midpoint uniform convexity and compactness for convergence in measure}

The key to the asymptotic midpoint uniform convexity of $X_{\mathsf{KW}}$ is the relative compactness of its unit ball with respect to the topology of convergence in measure. This is the only feature, together with the Daugavet Property, which will be needed. 

\begin{theo}[]\label{thm:main}
    Let $Y$ be a closed subspace of $L_1$ whose unit ball is relatively compact with respect to the topology of convergence in measure. Then, $Y$ is asymptotically midpoint uniformly convex.
\end{theo}

\begin{coro}[]
    The space $X_{\mathsf{KW}}$ is asymptotically midpoint uniformly convex, fails the Point of Continuity Property, and thus does not admit an equivalent norm that is asymptotically uniformly convex.
\end{coro}

In effect, the proof will show that a subspace of $L_1$ satisfying the assumption of Theorem \ref{thm:main} is asymptotically sprawling uniformly convex in the terminology of \cite{BLP}. This property, introduced by Basset, Lancien, and Proch\'azka in \cite{BLP}, was shown to be equivalent to asymptotic midpoint uniform convexity. Of course, one could easily present the proof by avoiding this equivalent reformulation of asymptotic midpoint uniform convexity, but since it seems that it is this specific reformulation of asymptotic midpoint uniform convexity that allowed the LLM to discover the elegant argument below, it is preferable to adopt the following presentation. According to \cite{BLP}, given a Banach space $(X,\norm{\cdot}_X)$, $\vep>0$, and $x\in B_X$, two sequences $\yn$ and $\zn$ in $B_X$ form an $\vep$-\emph{spider rooted at $x$} if
\begin{enumerate}
    \item  $x= \frac{y_n+z_n}{2}$, for all $n\ge1$,
    \item  $\inf_{n\neq m}\norm{y_n-y_m}_X\ge \vep$.
\end{enumerate}
A Banach space $(X,\norm{\cdot}_X)$ is \emph{asymptotically sprawling uniformly convex} if for every $\vep>0$ there is $\delta>0$ such that for every point $x\in B_X$ that is the root of an $\vep$-spider in $B_X$ one has $\norm{x}_X\leq 1-\delta$.

The fact that a Banach space is asymptotically sprawling uniformly convex if and only if it is asymptotically midpoint uniformly convex, and that their respective moduli are comparable, can be found in \cite[Proposition 6.16]{BLP}.

\begin{proof}[Proof of Theorem \ref{thm:main}]

Fix $\vep>0$ and let $\yn$ and $\zn$ in $B_Y$ form an $\vep$-spider rooted at $x\in B_Y$. Since $B_Y$ is relatively compact in measure, we can extract a common subsequence of $\yn$ and $\zn$ (still denoted the same) such that $\yn$ and $\zn$ converge almost everywhere to $y$ and $z$, respectively,
for some $y,z\in L_1$. Note that $\frac{y+z}{2}=x$ and it follows from Fatou's lemma that $\max\{\norm{y}_1, \norm{z}_1\}\le 1$. What remains to be shown is that the vectors $y$ and $z$ are uniformly far from the unit sphere to be able to control the norm of the root of the $\vep$-spider. The key ingredient is the following elementary measure-theoretic fact. 

\begin{clai}
\label{claim}
If a sequence $\{f_n\}_{n=1}^\infty$ in $L_1$ converges almost everywhere to some $f\in L_1$, then $\lim_{n\to\infty}( \norm{f_n}_1-\norm{f_n-f}_1)=\norm{f}_1$.
\end{clai}

This claim (which can be found in \cite{BrezisLieb1983}) is easily derived from the reverse triangle inequality and the dominated convergence theorem.

\smallskip\noindent

Since $\yn$ is bounded, we can pass to a further subsequence so that $\lim_{n\to\infty}\norm{y_n-y}_1 = \rho$ for some $\rho\in[0,2]$. It follows from the separation of $\yn$ and the triangle inequality that $\rho\ge\vep/2$. Now, observe that 
\[
 z_n-z=(2x-y_n)-(2x-y)=-(y_n-y),
\]
and hence we also have $\norm{z_n-z}_1\longrightarrow\rho$. 

\smallskip\noindent

To conclude the proof, we invoke Claim \ref{claim} twice, once for $\yn$ and once for $\zn$, so that  
\[
 \lim_{n\to\infty}(\norm{y_n}_1 - \norm{y_n-y}_1) = \norm{y}_1 
 \qquad 
\text{ and } 
\qquad
 \lim_{n\to\infty}(\norm{z_n}_1 - \norm{z_n-z}_1) = \norm{z}_1.
\]
It then follows from the estimates collected above that 
\[
 \norm{y}_1\le 1-\rho\le 1- \frac{\vep}{2} 
 \qquad\text{ and }
 \qquad
 \qquad \norm{z}_1\le 1-\rho\le 1- \frac{\vep}{2}.
\]

Recalling that $x=(y+z)/2$, the triangle inequality now gives $\norm{x}_1\le 1- \frac{\vep}{2}$.
\end{proof}

It follows from the proof of Theorem \ref{thm:main} and \cite[Proposition 6.16]{BLP} that the modulus of asymptotic midpoint uniform convexity of $X_{\mathsf{KW}}$ satisfies $\tilde{\delta}_{X_{\mathsf{KW}}}(t)\ge ct$ for some universal constant $c>0$.

\end{document}